\documentclass[12pt,psamsfonts]{amsart}
\usepackage{amssymb,amsmath}
\usepackage{amsthm}
\usepackage{mathrsfs}
\usepackage{enumerate, enumitem, indentfirst}
\usepackage[fleqn,tbtags]{mathtools}
\usepackage{color}
\usepackage{courier}
\usepackage{hyperref}
\usepackage{graphicx}
\usepackage{float}

\usepackage[a4paper,left=2.1cm,right=2.1cm,top=3.2cm,bottom=3.2cm]{geometry}
\usepackage{multicol}
\usepackage{xr}

\usepackage{multicol}
\usepackage{geometry}
\usepackage{marginnote}

\newtheorem*{theorem*}{Theorem}
\newtheorem*{corollary*}{Corollary}
\newtheorem{example}{Example}
\newtheorem*{remark*}{Remark}
\newtheorem*{proposition*}{Proposition}

 \newtheorem{theorem}{Theorem}
 \newtheorem{corollary}[theorem]{Corollary}

 \newtheorem{proposition}[theorem]{Proposition}

\newcommand{\ler}[1]{\left( #1 \right)}

\newcommand{\prob}{\mathcal{P}(\X)}
\newcommand{\meas}{\mathcal{M}(\X)}
\newcommand{\dirac}{\Delta(\X)}
\newcommand{\ffi}{\varphi}
\newcommand{\ws}{\mathcal{W}_p(\X)}
\newcommand{\W}{W_p^p}

\newcommand{\diagonal}{\mathrm{Diag}(\X^2)}
\newcommand{\cmn}{\mathcal{C}(\mu,\nu)}
\newcommand{\X}{\mathcal{X}}

\newcommand{\setx}{\{x\}}

\newcommand{\setn}{\{n\}}

\newcommand{\pic}{\pi^*}

\newcommand{\sfdx}{S_{f(\delta_x)}}
\newcommand{\sfdy}{S_{f(\delta_y)}}

\begin{document}
\title[On isometric embeddings of Wasserstein spaces -- the discrete case]{On isometric embeddings of Wasserstein spaces \\ -- the discrete case}

\author[Gy\"orgy P\'al Geh\'er]{Gy\"orgy P\'al Geh\'er}
\address{Gy\"orgy P\'al Geh\'er, Department of Mathematics and Statistics\\ University of Reading\\ Whiteknights\\ P.O.
Box 220\\ Reading RG6 6AX\\ United Kingdom}
\email{gehergyuri@gmail.com or G.P.Geher@reading.ac.uk\newline
\hspace{1.6cm} http://www.math.u-szeged.hu/\~{}gehergy}

\author[Tam\'as Titkos]{Tam\'as Titkos}
\address{Tam\'as Titkos, Alfr\'ed R\'enyi Institute of Mathematics\\ Hungarian Academy of Sciences\\ Re\'altanoda u. 13-15.\\
Budapest H-1053\\ Hungary\\ and BBS University of Applied Sciences\\ Alkotm\'any u. 9.\\
Budapest H-1054\\ Hungary}

\email{titkos.tamas@renyi.mta.hu \newline http://renyi.hu/\~{}titkos}

\author[D\'aniel Virosztek]{D\'aniel Virosztek}
\address{D\'aniel Virosztek, Institute of Science and Technology Austria \\ Am Campus 1 \\ 3400 Klos\-ter\-neuburg \\ Austria}

\email{daniel.virosztek@ist.ac.at \newline http://pub.ist.ac.at/\~{}dviroszt}



\thanks{Geh\'er was supported by the Leverhulme Trust Early Career Fellowship (ECF-2018-125), and also by the Hungarian National Research, Development and Innovation Office (Grant no. K115383); Titkos was supported by the Hungarian National Research, Development and Innovation Office - NKFIH (grant no. PD128374 and grant no. K115383), by the J\'anos Bolyai Research Scholarship of the Hungarian Academy of Sciences, and by the \'UNKP-18-4-BGE-3 New National Excellence Program of the Ministry of Human Capacities; Virosztek was supported by the ISTFELLOW program of the Institute of Science and Technology Austria (project code IC1027FELL01) and partially supported by the Hungarian National Research, Development and Innovation Office – NKFIH (grant no. K124152 and grant no. KH129601).}

\begin{abstract} The aim of this short paper is to offer a complete characterization of all (not necessarily surjective) isometric embeddings of the Wasserstein space $\mathcal{W}_p(\X)$, where $\X$ is a countable discrete metric space and $0<p<\infty$ is any parameter value. 
Roughly speaking, we will prove that any isometric embedding can be described by a special kind of $\X\times(0,1]$-indexed family of nonnegative finite measures. Our result implies that a typical non-surjective isometric embedding of $\ws$ splits mass and does not preserve the shape of measures.
In order to stress that the lack of surjectivity is what makes things challenging, we will prove alternatively that $\ws$ is isometrically rigid for all $0<p<\infty$.
\end{abstract}

\maketitle

\section{Introduction}
In order to understand a mathematical structure, it is natural to start with exploring its symmetries. In the case of metric spaces, this leads naturally to the study of isometries, and more generally, to the study of isometric embeddings.

As there is a crucial difference between these two types of maps and since our main result concerns this distinctions, let us fix their definitions here:
\begin{itemize}
	\item[--] an \emph{isometric embedding} is a map that preserves the distance between elements, which is not necessarily surjective,
	
	\item[--] an \emph{isometry} is a surjective map that preserves distance, which is sometimes called a \emph{surjective isometry}.

\end{itemize}
Of course, every isometry is an isometric embedding, but usually there are several isometric embeddings that are non-surjective, hence not isometries. The investigation of isometries has a long history, let us first introduce some of it, in particular those results which are most relevant to this paper.
One of the pioneers of these investigations was Banach, who presented the first characterizations of isometries on several classical spaces. Probably the most famous one is the Banach--Stone theorem, which says that isometries between $C(K)$-type spaces are induced by homeomorphisms of the underlying spaces followed by possible changes of sign in the function values on clopen (i.e., both closed and open) sets. When speaking only about this particular theorem, there is a vast literature of interesting generalizations, see the books \cite{fleming-jamison-1, fleming-jamison-2}. We mention briefly here only those in which \emph{metric spaces of probability measures} are involved, for a more detailed survey we refer the reader to \cite{Virosztek}. 
The first result in this direction is the description of all isometries of the non-linear space of all Borel probability measures over $\mathbb{R}$ with respect to the Kolmogorov--Smirnov distance, namely, Dolinar and Moln\'ar showed in \cite{Dolinar-Molnar} that they are always induced by homeomorphisms of $\mathbb{R},$ see also \cite{Molnar-CEJM}. 
Another metric appearing in statistics is the Kuiper metric, which was discussed recently in \cite{Kuiper} in full detail. Maybe the most relevant metrics in probability theory are those that metrise the weak converge like the L\'evy-, the L\'evy--Prokhorov-, and the Wasserstein metrics. Moln\'ar proved in \cite{Molnar-Levy} that the space of all Borel probability measures over $\mathbb{R}$ endowed with the L\'evy metric is isometrically rigid, i.e. every isometry is induced by an isometry of the underlying space $\mathbb{R}$.
Generalizing Moln\'ar's result, it has been proved in \cite{L-P} that the space of all Borel probability measures over a real and separable Banach space endowed with the L\'evy--Prokhorov metric is also isometrically rigid. 

The most important developments in our considerations have been done by Bertrand and Kloeckner in connection with the Wasserstein metric \cite{Bertrand-Kloeckner,Bertrand-Kloeckner 2,K-ann}. Besides that the Wasserstein distance metrises the weak convergence of probability measures, its importance also lies in its role in geometric investigations of metric spaces, for more information see \cite{LV,VRS,S,V,V-orig} and the references therein. Among many other deep results, Bertrand and Kloeckner have managed to prove that the Wasserstein space $\mathcal{W}_2(X)$ is isometrically rigid for any negatively curved geodesically complete Hadamard space $X$. 
We highlight also Kloeckner's results on the isometry group of $\mathcal{W}_2(\mathbb{R}^n)$, in particular, isometries of $\mathcal{W}_2(\mathbb{R}^n)$ are usually not induced by only one mapping of $\mathbb{R}^n$ -- unlike in the aforementioned results. Moreover, it turned out, which is one of the main results of \cite{K-ann}, that $\mathcal{W}_2(\mathbb{R})$ admits exotic isometries that does not even preserve the shape of measures. This result is pretty uncommon and it raises several questions, see Section 8 in \cite{K-ann}, and also the last section of this paper. One can imagine that things could become even more complicated when one drops the surjectivity condition and tries to characterize general isometric embeddings.

Concerning the characterization of isometric embeddings -- which is our main interest in this paper -- there are very few results in the literature. The main reason is that usually there is no hope for finding an elegant description of all isometric embeddings, because very wild maps can satisfy the distance preserving property. However, there are a few cases when the characterization of all isometric embeddings is known due to the underlying additional structure of the metric space -- let us now mention two of them. First, it is well-known that all isometric embeddings of strictly convex real Banach spaces are automatically affine, i.e. linear up to translation, so in order to characterize their isometric embeddings it is enough to explore the structure of linear norm-preserving operators. The latter was done by Lamperti in \cite{Lamperti} for $L^p$ spaces over $\sigma$-finite measure spaces when $1< p<\infty, p\neq 2$. Therefore a characterization of isometric embeddings of these $L^p$ spaces follows from Lamperti's result.
Another example is the famous Wigner's theorem about quantum mechanical symmetry transformations, which gives a characterization of isometric embeddings of the space of all rank-one projections of a Hilbert space with respect to the operator norm, see e.g. \cite{Ge-Wigner}, and also \cite{Ge-JFA} for a generalization. In this case there is a certain projective structure which can be taken advantage of.
\par
Our aim in this paper is to offer a complete description of all isometric embeddings of the Wasserstein space $\mathcal{W}_p(\X)$, where $\X$ is a countable discrete metric space and $0<p<\infty$ is any parameter. Roughly speaking, we will prove that any isometric embedding can be described by a special kind of $\X\times(0,1]$-indexed family of nonnegative finite measures, and vice versa, any such family of measures determines a distance preserving transformation. 
In the study of isometries it is a quite usual phenomenon that isometries also preserve some other underlying structure of the metric space. For instance, in \cite{Bertrand-Kloeckner, Dolinar-Molnar, Kuiper, L-P} and \cite{Molnar-Levy} it turned out -- as a consequence of the main theorems -- that every isometry is automatically affine, i.e. preserves convex combinations of measures; and in \cite{K-ann} that every isometry of $\mathcal{W}_2(\mathbb{R})$ preserves geodesically convex combinations of measures. 
Here we shall see that for a typical non-surjective isometric embedding of $\mathcal{W}_p(\X)$ this is not true anymore. Moreover, we will also see that a general isometric embedding splits mass (that is, it does not leave the set of Dirac masses invariant), and does not preserve the shape of measures (for a precise definition see the penultimate paragraph of this paper). To demonstrate that the real difficulty here is the lack of surjectivity, we will also prove alternatively that $\ws$ is isometrically rigid for all $0<p<\infty$.
It can be somewhat surprising that although the $W_{\infty}$-distance can be obtained as a limit of $W_p$ distances, the structure of its isometric embeddings is completely different.
\par
Let us now fix the setting and introduce the necessary definitions and notations to state and prove our results. Let $X\neq\emptyset$ be a countable set, and let $\rho:X^2\to\{0,1\}$ be the discrete metric, i.e., $\rho(x,y):=1$ if $x\neq y$ and $\rho(x,x):=0$ for all $x,y\in X$. To avoid trivialities, we will always assume that $X$ has at least two elements. We will denote the Polish space $(X,\rho)$ shortly by $\X$. The symbols $\prob$ and $\meas$ stand for the sets of probability measures and nonnegative finite measures on the power set of $\X$, respectively. For $\mu\in\meas$ and $T\subseteq\X$ the symbol $\mu|_T\in\meas$ stands for the restricted measure defined by
\begin{equation*}
\mu|_T(A):=\mu(A\cap T)\qquad\mbox{for all}~A\subseteq\X.
\end{equation*}
As usual, $\delta_x$ denotes the Dirac measure (or point mass) concentrated to $x\in \X$, i.e. $\delta_x(A)=1$ if $x\in A$, and $\delta_x(A)=0$ if $x\notin A$. The set of Dirac measures will be denoted by $\dirac=\{\delta_x\,|\,x\in \X\}$. It is obvious that each element of $\meas$ can be written as a weighted sum of Dirac measures, namely $\mu=\sum_{x\in \X}\mu(\setx)\cdot\delta_x$. The support of $\mu$ in this special case is just the set $S_{\mu}=\{x\in \X\,|\,\mu(\setx)>0\}$. As it is known, see Theorem 36.1 in \cite{ab}, the set of nonnegative finite measures on a fixed $\sigma$-algebra is a lattice. In this discrete case, the greatest lower bound $\mu\wedge\nu$ of two measures $\mu,\nu\in\meas$ can be calculated easily as
\begin{equation*}
\mu\wedge\nu = \sum_{x\in \X}\min\big\{\mu(\setx),\nu(\setx)\big\}\cdot\delta_x.
\end{equation*}
In order to introduce the Wasserstein space $\ws$, we need some more definitions. Since
\begin{equation*}
\mathcal{P}_p(\X):=\left\{\mu\in\prob\,\Big|\,\int_{\X}\rho(x,x_0)^p~d\mu(x)<\infty\quad\mbox{for some (and hence all)}~x_0\in\X\right\}
\end{equation*}
coincides with $\prob$ and $\rho^p=\rho$ for all parameter values $p \in (0, \infty),$ we will write always $\prob$ without indicating $p.$
We say that a Borel probability measure $\pi$ on $\X^2$ is a \emph{coupling} of $\mu,\nu\in \prob$ if it satisfies
\begin{equation}\label{coupling def}
\pi(A\times\X)=\mu(A)\qquad\mbox{and}\qquad\pi(\X\times A)=\nu(A)\qquad\quad\mbox{for all}~A\subseteq\X.
\end{equation}
The set of all couplings of $\mu$ and $\nu$ is denoted by $\cmn.$
The $p$-Wasserstein distance (which is indeed a metric on $\prob$) of $\mu$ and $\nu$ is defined by
\begin{equation}\label{wasser-metric}
W_p(\mu,\nu):=\left(\inf_{\pi\in\cmn}\int\limits_{\X^2} \rho(x,y)^p~d\pi(x,y)\right)^{1/p}
\end{equation}
if $p\geq 1$, and by
\begin{equation}\label{wasser-metric p<1}
W_p(\mu,\nu):=\inf_{\pi\in\cmn}\int\limits_{\X^2} \rho(x,y)^p~d\pi(x,y)
\end{equation}
if $p\in(0,1)$. For more details and historical comments we refer the reader to \cite[Chapter 6]{V}. Observe that $W_p(\mu,\nu)=W_1(\mu,\nu)$ holds for all $\mu,\nu\in \prob$ and $p\in(0,1)$, because $\rho^p=\rho$. The metric space $\ler{\prob,W_p}$ is denoted shortly by $\ws$ and is called the \emph{p-Wasserstein space over $\X$.}
We will call a transformation $f:\ws\to\ws$  an \emph{isometric embedding} if
\begin{equation*}W_p(f(\mu),f(\nu))=W_p(\mu,\nu)~\qquad\mbox{for all}~\mu,\nu\in\ws.
\end{equation*}
In the next section we present our main theorem, which is a complete characterization of isometric embeddings of $\ws$.

\section{Characterization of isometric embeddings}\label{main theorem}

First we express the Wasserstein distance in terms of the greatest lower bound of measures.
Note that this formula might be folklore, however, for the sake of completeness and for the convenience of the reader, we include it.

\begin{proposition} \label{prop}
	For any $1\leq p<\infty$ and $\mu, \nu \in \mathcal{W}_p(\X)$ we have
	\begin{equation}\label{W1-min}
		\W(\mu,\nu)=1-(\mu\wedge\nu)(\X).
	\end{equation}
\end{proposition}

\begin{proof}
The formula \eqref{W1-min} is obviously true for $\mu=\nu,$ hence we assume in the sequel that $\mu \neq \nu.$ For any coupling measure $\pi\in\cmn$ and point $x\in\X$ we notice that
\begin{equation*}
	\pi\big(\{(x,x)\}\big) \leq \min\big\{\pi(\setx\times\X),\pi(\X\times\setx)\big\} = \min\big\{\mu(\setx),\nu(\setx)\big\},
\end{equation*}
thus for the diagonal $\diagonal=\{(x,x)\,|\,x\in \X\}$ we conclude $\pi(\diagonal)\leq(\mu\wedge\nu)(\X)$. 
According to \eqref{wasser-metric} this inequality implies
\begin{equation*}
	\W(\mu,\nu)=\inf_{\pi\in\cmn}\int\limits_{\X^2\setminus\diagonal}1~d\pi(x,y)
	= \inf_{\pi\in\cmn} \left(1- \pi(\diagonal)\right)
	\geq 1-(\mu\wedge\nu)(\X),
\end{equation*}
therefore it is enough to construct a coupling $\pic\in\cmn$ for which $\pic(\diagonal)=(\mu\wedge\nu)(\X)$ holds.
To this aim, let us introduce some notation. For a signed measure $\theta$ on $\X,$ let $\theta_+$ denote the nonnegative measure given by $\theta_+\left(\{x\}\right)=\max \{0, \theta\left(\{x\}\right)\}$ for every $x \in \X.$ (In other words, $\theta_+$ is the nonnegative part of $\theta.$) Furthermore, for $\eta, \xi \in \meas,$ we denote by $\eta \otimes \xi$ the product measure of $\eta$ and $\xi,$ that is, $\left(\eta \otimes \xi\right) \left(A \times B\right)=\eta(A)\cdot\xi(B)$ for all $A,B \subseteq \X.$ Finally, let $i: \, \X \rightarrow \X^2; \, x \mapsto (x,x)$ be the diagonal embedding of $\X$ into $\X^2,$ and let us define $\pic$ by
\begin{equation*}
\pic:=i_{\#}\left(\mu \wedge \nu\right)+ \frac{\left(\mu-\nu\right)_+ \otimes \left(\nu-\mu\right)_+}{1-(\mu\wedge\nu)(\X)}
\end{equation*}
where $i_{\#}\left(\mu \wedge \nu\right)$ is the push-forward measure of $\mu \wedge \nu$ by the map $i$, that is,
\begin{equation*}
	\big(i_{\#}(\mu \wedge \nu)\big)(A)=(\mu \wedge \nu)\big(i^{-1}(A)\big) \quad \text{for all Borel sets} \; A \subseteq \X^2.
\end{equation*}
Note that $\left(\mu-\nu\right)_+(\X)=\left(\nu-\mu\right)_+(\X)=1-(\mu\wedge\nu)(\X),$ hence $\pic$ is indeed a probability measure on $\X^2.$ Moreover, it is also clear that $\pic(\diagonal)=(\mu \wedge \nu)(\X).$ To see that $\pic$ is a coupling of $\mu$ and $\nu,$ we first check that $\pic(A \times \X)=\mu(A)$ holds for every $A \subseteq \X$ by the following simple computation:
\begin{equation*}
\pic(A \times \X)=(\mu \wedge \nu)(A)+\left(\mu-\nu\right)_+(A) \cdot \frac{\left(\nu-\mu\right)_+(\X)}{1-(\mu\wedge\nu)(\X)}=(\mu \wedge \nu)(A)+\left(\mu-\nu\right)_+(A)=\mu(A).
\end{equation*}
As identity $\pic(\X\times A)=\nu(A)$ can be verified similarly, we conclude \eqref{W1-min} as it was claimed. 
\end{proof}

Note that for $0<p<1$ we have $W_p(\mu,\nu)=W_1(\mu,\nu)=1-(\mu\wedge\nu)(\X)$.
As one consequence of \eqref{W1-min} we can easily show that the Wasserstein space $\ws$ is isometrically rigid. That is, the isometry group of $\ws$ (consisting of all bijective distance-preserving transformations of $\ws$) is isomorphic to the isometry group of $\X$, which is just the permutation group of $\X$. 

\begin{corollary} \label{cor}
Let $p\in(0,\infty)$ and $f:\ws\to\ws$ be an isometry. 
Then there exists a bijection $\sigma:\X\to\X$ such that $f$ is a push-forward of $\sigma$, i.e.
\begin{equation}\label{pushforward}
	f(\mu)(\setx)=\mu(\{\sigma^{-1}(x)\})\qquad\mbox{for all}~x\in\X.
\end{equation}
Conversely, if a bijection $\sigma:\X\to\X$ is given, then \eqref{pushforward} defines an isometry.
\end{corollary}

\begin{proof}
By \eqref{W1-min} we have 
\begin{equation}\label{diszjunkt}
	S_{\mu}\cap S_{\nu}=\emptyset \;\iff\; \W(\mu,\nu)=1 \quad \text{for all} \; \mu,\nu\in\ws.
\end{equation}
Since $f$ is a bijection, the restriction $f|_{\dirac}$ is a bijection of $\dirac$. 
Indeed, assume that $x_1,x_2\in\sfdx$, $x_1\neq x_2$ happens, then we have $\W(f(\delta_x),\delta_{x_i})\neq 1$ ($i\in\{1,2\}$), therefore taking the inverse images gives
\begin{equation*}
\W(\delta_x,f^{-1}(\delta_{x_1}))\neq1,\qquad\W(\delta_x,f^{-1}(\delta_{x_2}))\neq1,\qquad\W(f^{-1}(\delta_{x_1}),f^{-1}(\delta_{x_2}))=1,
\end{equation*}
which is a contradiction by \eqref{diszjunkt}. Hence the $f$-image of every Dirac mass is a Dirac mass, and we argue similarly for the inverse $f^{-1}$.

Therefore the action of $f$ on $\dirac$ induces a bijection $\sigma$ on $\X$ defined by $f(\delta_x)=\delta_{\sigma(x)}$. Now, take any measure $\mu\in\prob$ and any $x \in \X$, and observe by \eqref{W1-min} that
\begin{equation*}
\mu(\{\sigma^{-1}(x)\})=\big(\mu\wedge\delta_{\sigma^{-1}(x)}\big)(\X)=\big(f(\mu)\wedge f(\delta_{\sigma^{-1}(x)})\big)(\X)=\big(f(\mu)\wedge\delta_x\big)(\X)=f(\mu)(\{x\}),
\end{equation*}
which proves one direction.
The reverse direction follows from the fact that the distance is invariant under permutations of the underlying space.\end{proof}

Now, we turn to the much more interesting problem of characterizing all isometric embeddings of $\ws$. 
In order to demonstrate that indeed quite wild maps can preserve the Wasserstein distance, we begin with two examples of isometric embeddings which are not affine, moreover, they split Dirac masses and therefore do not preserve the shape of measures.

\begin{example} 
\textup{Let us fix a parameter value $p\in(0,\infty)$, and let $\X = \{1,2,3,\dots\}$ be the set of all natural numbers endowed with the discrete metric. Define $f:\ws\to\ws$ as
\begin{equation*} 
	f\left(\sum\limits_{n\in \X}c_n\cdot\delta_n\right)=\sum\limits_{n\in \X}\big[\ln(1+c_n)\cdot \delta_{2n}+\big(c_n-\ln(1+c_n)\big)\cdot\delta_{2n+1}\big].
\end{equation*}
One can show by definition and using \eqref{W1-min} that this is indeed a non-surjective isometric embedding. Observe that the range of $f$ does not contain any Dirac mass. What happens here is roughly speaking the following: $f$ splits Dirac masses
\begin{equation*}
	f(\delta_n)=\ln2\cdot\delta_{2n}+(1-\ln2)\cdot\delta_{2n+1},
\end{equation*}
and redistributes weights. On one hand, if $n\neq m$, then $S_{f(\delta_n)}\cap S_{f(\delta_m)}=\emptyset$, thus $f$ induces a partition of $\X$, in fact, the support of $f(\mu)$ is always the disjoint union
\begin{equation}\label{suppeq}
	S_{f(\mu)} = \bigcup_{n\in S_{\mu}}S_{f(\delta_n)} = \bigcup_{n\in S_{\mu}} \{2n, 2n+1\}.
\end{equation}
On the other hand, we see that if $\mu(\setn)=c_n$, then $f(\mu)\big(\{2n,2n+1\}\big)=c_n$, and if $\mu(\setn)\leq\nu(\setn)$ then
\begin{equation*}
	f(\mu)|_{\{2n,2n+1\}}\leq f(\nu)|_{\{2n,2n+1\}}.
\end{equation*}
}
\end{example}

The second example is even wilder.

\begin{example}
\textup{
	Let $\X$ be again the set of all natural numbers endowed with the discrete metric.
	Also let $p_n$ denote the $n$th prime number.
	Define $f:\ws\to\ws$ as
	\begin{equation*} 
		f\left(\sum\limits_{n\in \X} c_n\cdot\delta_n\right)
		= \sum\limits_{n\in \X} \left( \sum_{j=1}^{N(c_n)} \frac{1}{2^j}\delta_{p_n^j} + \Bigg( c_n - \sum_{j=1}^{N(c_n)} \frac{1}{2^j} \Bigg) \delta_{p_n^{N(c_n)+1}} \right),
	\end{equation*}
	where for every $c\in [0,1)$ the non-negative integer $N(c)$ is defined by 
	\begin{equation*}
		\sum_{j=1}^{N(c)} \frac{1}{2^j} \leq c < \sum_{j=1}^{N(c)+1} \frac{1}{2^j}
	\end{equation*}
	and $N(1) = \infty$.
	Note that this map also shares similar properties mentioned in the previous example, which we do not repeat here.
	However, this map has a quite peculiar property, namely that every Dirac mass is mapped into an infinitely supported measure, but any other finitely supported measure is mapped into a finitely supported one. In particular, instead of \eqref{suppeq} in general we only have
	\begin{equation*}
		S_{f(\mu)} \subset \bigcup_{n\in S_{\mu}}S_{f(\delta_n)} = \bigcup_{n\in S_{\mu}} \left\{p_n^j \colon j=1,2,\dots\right\}.
	\end{equation*}
}
\end{example}

Now, we are in the position to claim and prove our main result.
We will see that in a particular sense every isometric embedding $f$ looks like the above examples:
$f$ will induce a partition on $\X$ and a family of nonnegative finite measures satisfying some special properties from which we can construct the whole action of $f$.

\begin{theorem} \label{thm}
Let $p\in(0,\infty)$ and $f:\ws\to\ws$ be an isometric embedding, i.e.,
\begin{equation}
W_p(f(\mu),f(\nu))=W_p(\mu,\nu)\qquad\mbox{for all}\quad\mu,\nu\in\ws.
\end{equation}
Then there exists a unique family of measures
\begin{equation}\label{Phi meas family}
\Phi:=\big(\ffi_{x,t}\big)_{x\in\X,t\in(0,1]}\in\meas^{\X\times(0,1]}
\end{equation}
that satisfies the properties
\begin{itemize}
\item[\textup{(a)}] for all $x\neq y$: $S_{\ffi_{x,1}}\cap S_{\ffi_{y,1}}=\emptyset$,
\item[\textup{(b)}] for all $x\in\X$ and $0 < t \leq 1$: $\ffi_{x,t}(\X)=t$,
\item[\textup{(c)}] for all $x\in\X$ and $0<s<t\leq 1$: $\ffi_{x,s}\leq\ffi_{x,t}$,
\end{itemize}
and that generates $f$ in the following sense 
\begin{equation}\label{f}
f(\mu)=\sum_{x\in S_{\mu}}\ffi_{x,\mu(\setx)}\qquad\mbox{for all}\quad\mu\in\ws.
\end{equation}
Conversely, every $\X\times(0,1]$-indexed family of measures satisfying \textup{(a)--(c)} generates an isometric embedding via the formula \eqref{f}.
\end{theorem}

\begin{proof}
As $\mathcal{W}_p(\X)$ spaces are indistinguishable from $\mathcal{W}_1(\X)$ whenever $p\in(0,1)$, we will always assume for the sake of simplicity that $p\geq1$.
Let $f$ be a fixed isometric embedding, our first step is to construct the family $\Phi$. 
Observe that by \eqref{W1-min} we have
\begin{equation}\label{inf megorzes}
	\big(f(\mu)\wedge f(\nu) \big)(\X)=\big(\mu\wedge\nu\big)(\X)\qquad\mbox{for all}~\mu,\nu\in\prob,
\end{equation}
in particular, $f$ preserves the disjointness of supports in both directions:
\begin{equation}\label{diszjunkt megorzes}
	S_\mu \cap S_\nu = \emptyset \;\;\iff\;\; S_{f(\mu)} \cap S_{f(\nu)} = \emptyset \qquad\mbox{for all}~\mu,\nu\in\prob.
\end{equation}
As a special case we obtain $\sfdx\cap\sfdy=\emptyset$ whenever $x\neq y$.
Since every $\mu\in\prob$ can be written in the form $\mu=\sum\limits_{x\in\X}\mu(\setx)\cdot\delta_x$, and the sets $\sfdx$ $(x\in\X)$ are pairwise disjoint, it is a natural idea to define a map $F:\dirac_\leq\to\meas$ on $\dirac_\leq := \{t\cdot\delta_x\,|\,t\in(0,1],\,x\in\X\}$ for which $F|_{\dirac}=f|_{\dirac}$ holds, and which is compatible with $f$. By compatibility we mean that if $0<t=\mu(\setx)=\nu(\setx)$, then 
\begin{equation*}
	f(\mu)|_{\sfdx}=f(\nu)|_{\sfdx}=F(t\cdot\delta_x)
\end{equation*} 
is satisfied.
Of course it is not clear at this point why $F$ should be well-defined.
In order to see that, it is enough to prove that if $0<s=\mu(\setx)\leq\nu(\setx)=t$ holds for $\mu,\nu\in\prob$, then $f(\mu)|_{\sfdx}\leq f(\nu)|_{\sfdx}$. 
\begin{figure}[H]
\centering
\includegraphics[width=90mm]{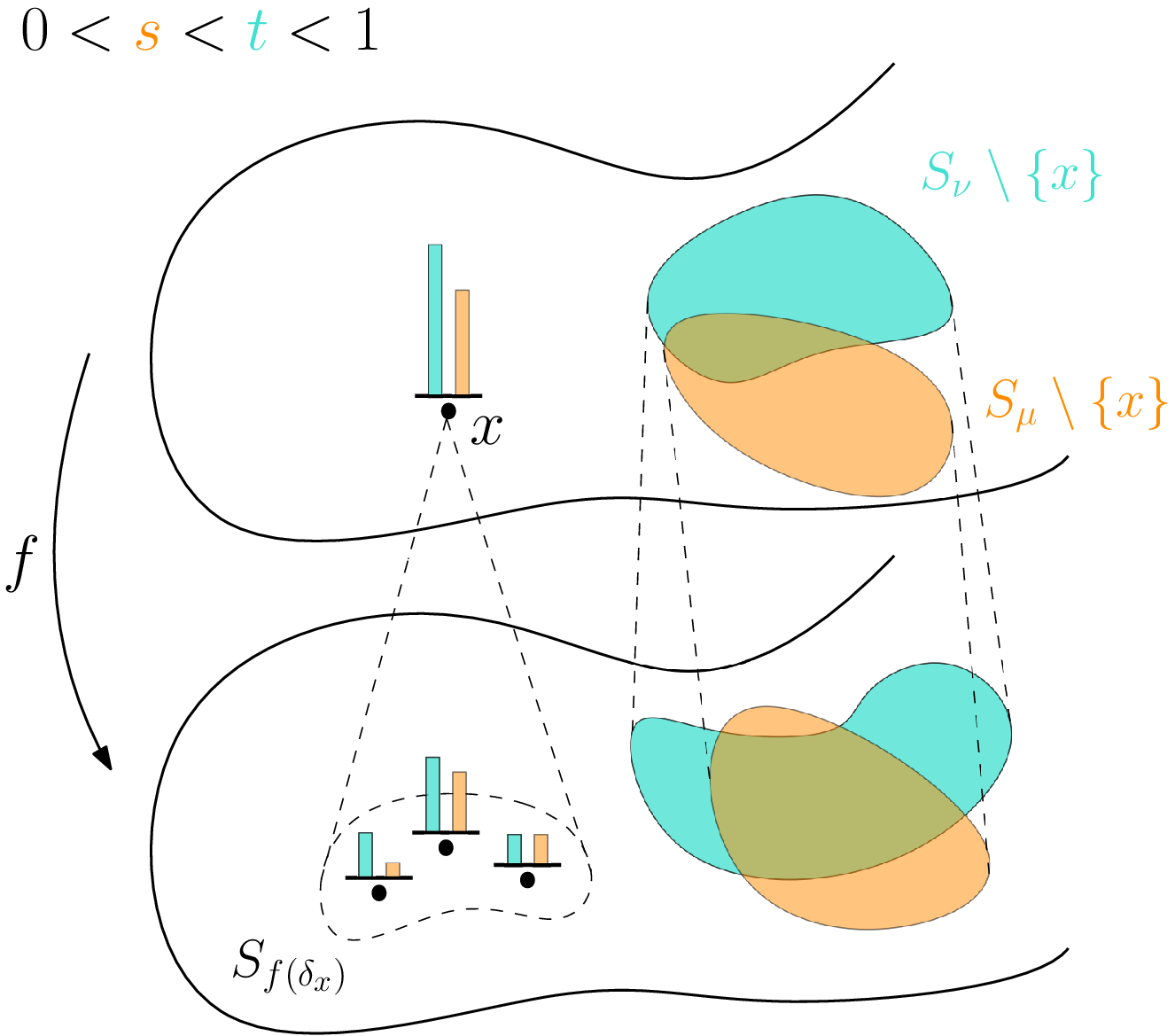}
\caption{$f(\mu)|_{\sfdx}\leq f(\nu)|_{\sfdx}$. \label{overflow}}
\end{figure}

We start by understanding the structure of $f(\mu)$. First observe that by \eqref{inf megorzes}
\begin{equation}\label{mux=fmusfdx}
	\mu(\setx)=\big(\delta_x\wedge\mu\big)(\X)=\big(f(\delta_x)\wedge f(\mu)\big)(\X)=\big(f(\delta_x)\wedge f(\mu)\big)(\sfdx)\leq f(\mu)(\sfdx)
\end{equation}
holds for all $x\in\X$. 
Since the supports of $f(\delta_x)$ are pairwise disjoint, we easily obtain that
\begin{equation*}
	\mu(\setx)=f(\mu)(\sfdx)=f(\mu)|_{\sfdx}(\X) \quad \text{for all}\;\; x\in\X,
\end{equation*}
and hence that $S_{f(\mu)}\subseteq\bigcup\limits_{x\in S_\mu}\sfdx$.
Now, for arbitrary probability measures $\mu$ and $\nu \in \prob$ we calculate
\begin{align*}
\sum\limits_{x\in S_{\mu}\cap S_{\nu}}\min\big\{\mu(\setx),\nu(\setx)\big\}&=\big(\mu\wedge\nu\big)(\X)=\big(f(\mu)\wedge f(\nu)\big)(\X)\\
&=\Big[\Big(\sum_{x\in S_{\mu}}f(\mu)|_{\sfdx}\Big)\wedge\Big(\sum_{x\in S_{\nu}}f(\nu)|_{\sfdx}\Big)\Big](\X)\\
&=\sum\limits_{x\in S_{\mu}\cap S_{\nu}}\Big(f(\mu)|_{\sfdx}\wedge f(\nu)|_{\sfdx}\Big)(\X)\\
&\leq\sum\limits_{x\in S_{\mu}\cap S_{\nu}}\min\Big\{ (f(\mu)|_{\sfdx}(\X), f(\nu)|_{\sfdx}(\X)\Big\}\\
&=\sum\limits_{x\in S_{\mu}\cap S_{\nu}}\min\big\{\mu(\setx),\nu(\setx)\big\},
\end{align*}
which forces the following for all $x\in S_{\mu}\cap S_{\nu}$:
\begin{equation}\label{forced}
\Big(f(\mu)|_{\sfdx}\wedge f(\nu)|_{\sfdx}\Big)(\X)=\min\Big\{ (f(\mu)|_{\sfdx}(\X), f(\nu)|_{\sfdx}(\X)\Big\}.
\end{equation}

At this point let us note the trivial fact that the partial ordering on the set of all non-negative finite measures can be expressed by the measure of $\X$. Indeed, we have
\begin{equation}\label{inftrick}
	\alpha\leq\beta\quad\iff\quad\alpha(\X)\leq\big(\alpha\wedge\beta)(\X) \qquad \text{for all}\; \alpha,\beta\in\meas.
\end{equation}
Therefore, if $0<s=\mu(\setx)\leq\nu(\setx)=t$ holds for a point $x\in S_{\mu}\cap S_{\nu}$, then by \eqref{forced} and \eqref{inftrick} we obtain
\begin{equation*}
	f(\mu)|_{\sfdx} \leq f(\nu)|_{\sfdx},
\end{equation*}
as was claimed.

Now, we define the family $\Phi$ by 
\begin{equation}\label{csalad}
	\ffi_{x,t} := F(t\cdot\delta_x) \qquad(x,t)\in\X\times(0,1].
\end{equation}
From here (a)--(c) is obvious, and hence \eqref{f} is also straightforward.

Finally, let us prove the reverse statement. 
Assume that $\Phi=\big(\ffi_{x,t}\big)_{x\in\X,t\in(0,1]}\in\meas^{\X\times(0,1]}$ is a family satisfying (a)--(c), then the following calculation shows that the map defined by \eqref{f} is indeed an isometric embedding:
\begin{align*}
\W(f(\mu),f(\nu))&=1-[f(\mu)\wedge f(\nu)](\X)\\
&=1-\Big[\Big(\sum_{x\in S_{\mu}}\ffi_{x,\mu(\setx)}\Big)\wedge\Big(\sum_{x\in S_{\nu}}\ffi_{x,\nu(\setx)}\Big)\Big](\X)\\
&=1-\sum\limits_{x\in S_{\mu}\cap S_{\nu}}\Big[\ffi_{x,\mu(\setx)}\wedge\ffi_{x,\nu(\setx)}\Big](\X)\\
&=1-\sum\limits_{x\in S_{\mu}\cap S_{\nu}}\ffi_{x,\min\big\{\mu(\setx),\nu(\setx)\big\}}(\X)\\
&=1-\sum\limits_{x\in S_{\mu}\cap S_{\nu}}\min\big\{\mu(\setx),\nu(\setx)\big\}=1-\big(\mu\wedge\nu)(\X)=\W(\mu,\nu).
\end{align*}
\end{proof}

Let us now make a few remarks on our main result. 
As we noted earlier, in the theory of isometries it is a quite typical phenomenon that isometries also preserve some other structure of the metric space.
For instance, in \cite{Bertrand-Kloeckner, Dolinar-Molnar, Kuiper, L-P, Molnar-Levy} all isometries were induced by either an isometry or a homeomorphism of the underlying space (which were either $\mathbb{R}$, or a real separable Banach space, or a negatively curved geodesically complete Hadamard space).
In particular, from these results it follows that isometries of those spaces are automatically affine, hence for any isometry $\phi$ once we know the $\phi$-images of two measures, the $\phi$-images of their convex combinations are also determined.
Furthermore, it follows easily that the action of $\phi$ on Dirac masses determines its action on the whole metric space.
However, this is definitely not the case for general isometric embeddings of $\ws$ (although it is the case for its isometries as we saw in Corollary \ref{cor}). 
To demonstrate this, assume we know the $f$-images of two Dirac masses $\delta_x$ and $\delta_y$, $x\neq y$, and consider an arbitrary third measure supported on $\{x,y\}$: $\mu = c\cdot\delta_x + (1-c)\cdot\delta_y \in\prob\setminus\{\delta_x,\delta_y\}$.
Then the only thing we can conclude for $f(\mu)$ is that it has the form $\mu_x + \mu_y$ where $\mu_x \leq f(\delta_x), \mu_x(\X) = c$ and $\mu_y \leq f(\delta_y), \mu_y(\X) = 1-c$ (cf. the previous two examples).
Moreover, it is not too hard to see that for any such choice of $\mu_x$ and $\mu_y$ we can construct an isometric embedding $f$ such that $f(\mu) = \mu_x + \mu_y$.

As was pointed out earlier, isometries of $W_2(\mathbb{R})$ automatically preserve geodesic convex combinations, however, an isometry of $W_2(\mathbb{R})$ is usually not determined by its action on Dirac masses, unlike in the previous paragraph.
On the other hand, it is easy to see from the results of \cite{K-ann} that any isometry of $W_2(\mathbb{R})$ is determined by its action on the set $\Delta_2'$ of all probability measures with two-point supports. It is interesting to note that any isometric embedding of $\ws$ is also determined by its action on the set $\Delta_2'(\X) := \left\{\mu\in\prob\colon \# S_\mu = 2\right\}$, although we have more flexibility than in the case of $W_2(\mathbb{R})$.

We also remark that if the underlying space $\X$ is finite, then isometric embeddings of $\ws$ are automatically surjective, hence they are implemented by a permutation of $\X$ according to Corollary \ref{cor}.

Next, we would like to point out that the limiting case when $p=\infty$ is very different. 
Let $\mathcal{W}_{\infty}(\X)$ be the set $\prob$ endowed with the metric $W_{\infty}(\mu,\nu)=\lim\limits_{p\to\infty}W_p(\mu,\nu)$ (see Proposition 3 and the references in \cite{inftyref}). 
Then according to \eqref{W1-min}, we obtain that $W_{\infty}(\mu,\nu)=1$ if and only if $\mu\neq\nu$, thus $\mathcal{W}_{\infty}(\X)$ is the set $\prob$ endowed with the discrete metric.
Therefore any injective function $f:\mathcal{W}_{\infty}(\X)\to\mathcal{W}_{\infty}(\X)$ is an isometric embedding, and any bijection is an isometry.
\par
We close this paper with some remarks about Kloeckner's questions. First we recall that for a given Polish space $M,$ a map $f:\mathcal{W}_p(M)\to\mathcal{W}_p(M)$ is called \emph{shape preserving} if for all $\mu\in \mathcal{W}_p(M)$ there exists an isometry $\psi_{\mu}:M\to M$ (depending on $\mu$), such that $f(\mu)$ is the push-forward measure of $\mu$ by the map $\psi_{\mu}.$
Note that the isometries discussed in \cite{Bertrand-Kloeckner, L-P, Molnar-Levy} are all shape preserving.
A distance preserving map is called \emph{exotic} if it is not shape preserving. 
Note that the isometries discussed in \cite{Dolinar-Molnar, Kuiper} are exotic in this particular sense, although each of them is induced by a transformation of $\mathbb{R}$, unlike in \cite{K-ann}.
Now we are in the position to formulate Kloeckner's questions \cite[Question 8.1 and Question 8.2]{K-ann}:
\begin{itemize}
\item[\textbf{Q1}] Is there a Polish (or Hadamard) space $M \neq\mathbb{R}$ such that $\mathcal{W}_2(M)$ admits exotic isometries?
\item[\textbf{Q2}] Is there a Polish space $M$ whose Wasserstein space $\mathcal{W}_2(M)$ possesses an isometry that does not preserve the set of all Dirac masses?
\end{itemize}
We cannot answer exactly these questions because we are dealing with isometric embeddings, and not with bijections. However, according to our main result (also, it is worth to revisit the previous two examples), there are many isometric embeddings $f:\mathcal{W}_2\ler{\X}\to\mathcal{W}_2\ler{\X}$ such that (i) f is exotic, and (ii) $f$ does not preserve the set of Dirac masses. Moreover, the parameter value $p=2$ is crucial as well, because as we have seen already, any bijection of $\mathcal{W}_{\infty}(\X)$ is an isometry and thus both (i) and (ii) can happen even in the bijective case, if $p=\infty$.
\par
Finally, concerning the above questions, let us remark that in a forthcoming paper we will discuss a Polish space $M$ such that for all $1<p<\infty$ the space $\mathcal{W}_p(M)$ is isometrically rigid, but the space $\mathcal{W}_1(M)$ has isometries that do not even preserve the set of all Dirac measures.

\section*{Acknowledgements}

We are grateful to L\'aszl\'o Erd\H{o}s and the referees for their helpful suggestions on the presentation of the paper.

\end{document}